\documentclass[11pt,a4paper,twoside]{article}
\usepackage{array,latexsym,amsmath,amssymb}
\usepackage{artmods,macros}
\usepackage[utf8]{inputenc}
\usepackage{xcolor}
\usepackage{cite}

\input{macros}

\makeatletter
\newcommand{\optval}{\mathop{\operator@font optval}}

\makeatother

\def\Rp{{R_+}}
\def\Rz{{R_0}}
\def\Cp{{C_+}}
\def\Cz{{C_0}}

\def\Rk{{R_k}}
\def\Ck{{C_k}}

\def\Rkm{{R_{k-1}}}
\def\Ckm{{C_{k-1}}}
\def\Rkp{{R_{k+1}}}
\def\Ckp{{C_{k+1}}}

\begin{document}

\title{ON THE EXISTENCE OF A SHORT PIVOTING SEQUENCE FOR A LINEAR PROGRAM}

\markboth{On decomposition of a matrix by convex combinations}{}

\author{Anders FORSGREN\thanks{\footAF} \and Fei WANG\thanks{\footFW}}

\def\footAF{Optimization and Systems Theory, Department of
  Mathematics, KTH Royal Institute of Technology, SE-100 44 Stockholm,
  Sweden ({\tt andersf@kth.se}).  Research supported by the Swedish
  Research Council (VR).}

\def\footFW{Optimization and Systems Theory, Department of
  Mathematics, KTH Royal Institute of Technology, SE-100 44 Stockholm,
  Sweden ({\tt fewa@kth.se}).}


\maketitle\thispagestyle{empty}

\begin{abstract}
  Pivoting methods are of vital importance for linear programming, the
  simplex method being the by far most well-known. In this paper, a
  primal-dual pair of linear programs in canonical form is
  considered. We show that there exists a sequence of pivots, whose
  length is bounded by the minimum dimension of the constraint matrix,
  such that the pivot creates a nonsingular submatrix of the
  constraint matrix which increases by one row and one column at each
  iteration. Solving a pair of linear equations for each of these
  submatrices generates a sequence of optimal solutions of a
  primal-dual pair of linear programs of increasing dimensions,
  originating at the origin. The optimal solutions to the original
  primal-dual pair of linear programs are obtained in the final step.

  It is only an existence result, we have not been able to generate
  any rules based on properties of the problem to generate the
  sequence. The result is obtained by a decomposition of the final
  basis matrix.
\end{abstract}

\section{Introduction}

Pivoting methods for linear programming are based on solving a
sequence of linear system of equations determined by a square
submatrix of the constraint matrix, that typically changes by one
column and/or one row in between iterations. The simplex
method~\cite{Dan63} is probably by far the most well-known, but we
also want to mention criss-cross methods~\cite{Ter85,FT97,BonM03} and
Lemke's method~\cite{lemke1954}. 

We consider a primal-dual pair of linear programs in canonical form
\[
(PLP) \quad
 \begin{array}{ll}
 \minimize{x\in\Re^n} & c\T x  \\ [4pt]
 \subject           & A x \ge b, \\
                    & x \ge 0,
 \end{array}
\quad\quad
(DLP) \quad
 \begin{array}{ll}
 \maximize{y\in\Re^m} & b\T y  \\ [4pt]
 \subject           & A\T y \le c, \\
                    & y \ge 0.
 \end{array}
\]
We show that if $(PLP)$ and $(DLP)$ are both feasible, then there
exists a nonnegative integer $r$, with $r\le\min\{m,n\}$, and a
sequence of pivots $(i_k,j_k)$, $k=1,\dots,r$, which generate sets of
row indices $R_k=\cup_{l=1}^k\{i_l\}$ and columns indices
$C_k=\cup_{l=1}^k\{j_l\}$, with $i_{k+1}\in\{1,\dots,m\}\backslash
\Rk$ and $j_{k+1}\in\{1,\dots,n\}\backslash \Ck$, such that
$A_{R\k\Ck}$ is nonsingular, and if $x_\Ck$ and $y_\Rk$ are computed
from
\[
A_{\Rk\Ck} x_\Ck = b_\Rk, \quad A_{\Rk\Ck}^T y_\Rk = c_\Ck, 
\]
they are nonnegative, and therefore optimal to $(PLP_k)$ and $(DLP_k)$
respectively, where
\[
(PLP_k) \quad
 \begin{array}{ll}
 \minimize{x_\Ck\in\Re^k} & c_\Ck^T x_\Ck\drop  \\ [4pt]
 \subject           & A_{\Rk\Ck}\drop x_\Ck\drop \ge b_\Rk\drop, \\
                    & x_\Ck\drop \ge 0,
 \end{array}
\quad\quad
(DLP_k) \quad
 \begin{array}{ll}
 \maximize{y_\Rk\in\Re^k} & b_\Rk^T y_\Rk\drop  \\ [4pt]
 \subject           & A_{\Rk\Ck}^T y_\Rk\drop \le c_\Ck\drop, \\
                    & y_\Ck\drop \ge 0.
 \end{array}
\]
Finally, $x_{C_r}\drop$ and $y_{R_r}\drop$ are not only optimal to
$(PLP_r)$ and $(DLP_r)$, but together with $x_j=0$ for
$j\in\{1,\dots,n\}\backslash C_r$ and $y_i=0$ for
$i\in\{1,\dots,m\}\backslash R_r$ they also give optimal solutions to
$(PLP)$ and $(DLP)$ respectively. We refer to such a sequence of
pivots as a \emph{short} sequence of pivots. The existence of this
short sequence of pivots is shown by a decomposition of the optimal
basis matrix.

We also give a related result for a slightly more structured linear
program, which is a min-max problem for a given $m\times
n$ matrix $M$, formulated as the following primal-dual pair of linear
programs
\begin{displaymath}
(P)\quad
  \begin{array}{ll}
     \minimize{u\in\Re^n,\alpha\in\Re}    & \alpha \\
     \subject & Mu + e \alpha \ge 0, \\
              & e\T u = 1, \\
              & u\ge 0,
  \end{array}
\qquad\quad
(D)\quad
  \begin{array}{ll}
     \maximize{v\in\Re^m,\beta\in\Re} & \beta \\
     \subject & M\T v +e \beta \le 0, \\
         & e\T v = 1, \\
         & v \ge 0.
  \end{array}
\end{displaymath}

For the general linear program, we cannot relate the short sequence of
pivots to monotonicity in objective function value, whereas this can
be done for the min-max problem. The difference is that $(P)$ and
$(D)$ are both defined on the unit simplex, and they are both always
feasible.

The results are straightforward, but to the best of our knowledge, our
result on the existence of a sequence of pivots of length at most
$\min\{m,n\}$ improves on what is known. Fukuda, Lühti and
Namiki~\cite{FLN97} and Fukuda and Terlaky~\cite{FT99} show that there
is a sequence of pivots of length bounded by at most $m+n$ leading to
the optimal solution. It should be pointed out that our existence
result does not automatically give a method with better worst-case
complexity than enumeration of all potential basis matrices. We have
not been able to use the information to give rules based on global
information that makes use of the short sequence of pivots.

\section{Existence of a short pivoting sequence for a linear program}

We will refer to a nonsingular square submatrix of $A$ as a
\emph{basis matrix}. If $\Rp$ and $\Cp$ denote the row and column
indices of $A$ that define the basis matrix, the basis matrix is
referred to as $A_{\Rp\Cp}$. If $\Rz$ and $\Cz$ denote the remaining
row and column indices, the primal and dual \emph{basic solutions}
associated with $A_{\Rp\Cp}$ are uniquely given by
\begin{dbleqnarray*}
A_{\Rp\Cp}\drop x_{\Cp}\drop & = & b_\Rp\drop, & A_{\Rp\Cp}^T y_\Rp\drop & =
& c_\Cp\drop,  \\ 
       x_\Cz\drop & = & 0, &        y_\Rz\drop & = & 0.
\end{dbleqnarray*}
The primal-dual pair of basic solutions given by the basis matrix is
optimal to $(PLP)$ and $(DLP)$ respectively if and only if the
solutions are feasible to $(PLP)$ and $(DLP)$ respectively, i.e.,
\begin{dbleqnarray*}
A_{\Rz\Cp} x_\Cp\drop & \ge & b_\Rz\drop, & A_{\Rp\Cz}^T y_\Rp\drop & \le &
c_\Cz\drop,  \\ 
       x_\Cp\drop & \ge & 0, &        y_\Rp\drop & \ge & 0.
\end{dbleqnarray*}
If $(PLP)$ and $(DLP)$ are both feasible, then there exists at least
one basis matrix which gives a primal-dual optimal pair of basic
solutions. This well-known result is summarized in the following
lemma.

\begin{lemma}[Existence of optimal basic feasible 
  solution]\label{lem:partitionLP} 
  Assume that both $(PLP)$ and $(DLP)$ are feasible. Then, there is a
  partitioning of the row indices of $A$ into two sets $\Rp$ and
  $\Rz$, and a partitioning of the column indices of $A$ into two sets
  $\Cp$ and $\Cz$, so that $\abs{\Rp}=\abs{\Cp}$, and associated with
  the resulting matrix
\[
\mtx{cc}{ A_{\Rp \Cp} & A_{\Rp \Cp} \\
          A_{\Rz \Cp} & A_{\Rz \Cz} },
 \]
 the submatrix $A_{\Rp \Cp}$ is nonsingular, and there are vectors $x$
 and $y$ for which
\begin{dbleqnarray*}
A_{\Rp\Cp}\drop x_{\Cp}\drop & = & b_\Rp\drop, & A_{\Rp\Cp}^T y_\Rp\drop & =
& c_\Cp\drop,  \\ 
A_{\Rz\Cp}\drop x_\Cp\drop & \ge & b_\Rz\drop, & A_{\Rp\Cz}^T y_\Rp\drop & \le &
c_\Cz\drop,  \\ 
       x_\Cp\drop & \ge & 0, &        y_\Rp\drop & \ge & 0, \\
       x_\Cz\drop & = & 0, &        y_\Rz\drop & = & 0,
\end{dbleqnarray*}
hold. These vectors $x$ and $y$ are optimal solutions to $(PLP)$ and
$(DLP)$ respectively.
\end{lemma}

\begin{proof}
Proofs are typically given for standard form of a linear program,
e.g., \cite[Theorem 3.4]{MR3100219}. This can be achieved by adding
slack variables to $(PLP)$, from which the result follows.
\end{proof}

Our concern is to decompose the basis matrix by eliminating one row
and one column at a time. The following lemma gives the basis for such
an elimination of row $i$ and column $j$ for a given set of row indices
$R$ and column indices $C$.

\begin{lemma}\label{lem:invLP}
  Consider problems $(PLP)$ and $(DLP)$. Let $R$ denote a set of row
  indices of $A$ and let $C$ denote a set of column indices of
  $A$. Assume that
\begin{equation}\label{eq:deltaxy}
A_{RC} x_C = b_R, \quad A_{RC}^T y_R = c_C, \quad
A_{RC} \Deltait x_C = e_i, \quad A_{RC}^T \Deltait y_R = e_j,
\end{equation}
where $e_i$ and $e_j$ are the $i$th and $j$th unit vectors of
dimensions $\abs{R}$ and $\abs{C}$ respectively. Then,
\begin{subequations}
\begin{eqnarray}
c_C^T x_C & = & b_R^T y_R, \\
c_C^T \Deltait x_C & = & e_i^T y_R, \\
b_R^T \Deltait y_R & = & e_j^T x_C, \\
e_j^T \Deltait x_C & = & e_i^T \Deltait y_R.
\end{eqnarray}
\end{subequations}
In addition, assume that $e_j^T \Deltait x_C = e_i^T \Deltait y_R\ne
0$. Then there are unique scalars $s_i$ and $t_j$ such that
\begin{equation}\label{eqn-zericomp}
e_j^T (x_C+ s_i \Deltait x_C) = 0, \quad
e_i^T (y_R+ t_j \Deltait y_R) = 0,
\end{equation}
given by
\begin{equation}\label{eqn-step}
s_i=-\frac{e_j^T x_C}{e_j^T \Deltait x_C}, \quad t_j = -\frac{e_i^T
  y_R}{e_i^T \Deltait y_R}. 
\end{equation}
Furthermore,
\begin{equation}\label{eqn-costsame}
c_C^T ( x_C + s_i \Deltait x_C ) = b_R^T ( y_R + t_j \Deltait y_R ).
\end{equation}
\end{lemma}

\begin{proof}
We obtain
\begin{eqnarray*}
c_C^T x_C & = & y_R^T A_{RC} x_C = y_R^T b_R = b_R^T y_R, \\
c_C^T \Deltait x_C & = & y_R^T A_{RC} \Deltait x_C = y_R^T e_i = e_i^T
y_R, \\
b_R^T \Deltait y_R & = & x_C^T A_{RC}^T \Deltait y_R = x_C^T e_j = e_j^T x_C, \\
e_j^T \Deltait x_C & = & \Deltait y_R^T A_{RC} \Deltait x_C =
\Deltait y_R^T e_i = e_i^T \Deltait y_R.
\end{eqnarray*}
To show the final results, if $e_j^T \Deltait x_C = e_i^T \Deltait
y_R\ne 0$, then the values of $s_i$ and $t_j$ given by
(\ref{eqn-step}) follow immediately. From these values of $s_i$ and
$t_j$, we obtain 
\begin{eqnarray*}
0 & = & e_j^T (x_C+ s_i \Deltait x_C) = \Deltait y_R^T A_{RC} 
(x_C+ s_i \Deltait x_C) = b_R^T \Deltait y_R + s_i \Deltait y_R^T
A_{RC} \Deltait x_C, \\
0 & = & e_i^T (y_R+ t_j \Deltait y_R) = \Deltait x_C^T A_{RC}^T 
(y_R+ t_j \Deltait y_R) = c_C^T \Deltait x_C + t_j \Deltait x_C^T
A_{RC}^T \Deltait y_R.
\end{eqnarray*}
A combination of these equations gives
\begin{equation}\label{eqn-costdiffsame}
t_j b_R^T \Deltait y_R = s_i c_C^T \Deltait x_C = 
- s_i t_j \Deltait x_C^T A_{RC}^T \Deltait y_R.
\end{equation}
A combination of $b_R^T y_R = c_C^T x_C$ and (\ref{eqn-costdiffsame})
gives (\ref{eqn-costsame}), as required.
\end{proof}

This result may now be used to reduce the dimension of the basis
matrix by one row and one column, while maintaining primal and dual
optimality to the reduced problem.

\begin{lemma}\label{lem-reducedimLP}
  Consider problems $(PLP)$ and $(DLP)$. Let $\Rk$ denote a set of row
  indices of $A$ and let $\Ck$ denote a set of column indices of $A$
  such that $\abs{\Rk}=\abs{\Ck}=k$, with $k\ge 2$. Assume that
  $A_{\Rk\Ck}$ is nonsingular, and assume that
\begin{displaymath}
A_{\Rk\Ck}\drop x_\Ck\drop  = b_\Ck\drop, 
\qquad
A_{\Rk\Ck}^T y_\Rk\drop = c_\Ck\drop, 
\end{displaymath}
where $x_\Ck\drop \ge 0$ and $y_\Rk\drop \ge 0$. Then, there is a row
index $i_k$, with $i_k\in\Rk$, and a column index $j_k$, with
$j_k\in\Ck$, such that $A_{\Rkm\Ckm}$ is nonsingular, where
$\Rkm=\Rk\backslash\{i_k\}$ and
$\Ckm=\Ck\backslash\{j_k\}$. Furthermore, it holds that
\begin{displaymath}
A_{\Rkm\Ckm} x_\Ckm\drop  = b_\Ckm\drop, 
\qquad
A_{\Rkm\Ckm}^T y_\Rkm\drop = c_\Ckm\drop, 
\end{displaymath}
for $x_\Ckm\drop\ge0$, $y_\Rkm\drop\ge0$.
\end{lemma}

\begin{proof}
We may apply Lemma~\ref{lem:invLP} for $R=\Rk$ and $C=\Ck$. The
quantities $x_\Ck$, $y_\Rk$, $\Deltait x_\Ck$ and $\Deltait y_\Rk$ are
well defined since $A_{\Rk\Ck}$ is nonsingular.

First, assume that $y_i=0$ for some $i\in\Rk$. Let $i_k=i$. Compute
$\Deltait x_\Ck$ as in Lemma~\ref{lem:invLP} for this $i$. If
$\Deltait x_\Ck\not\ge 0$, we may compute the most limiting positive
step for maintaining nonnegativity of $x_\Ck + s \Deltait x_\Ck$,
i.e.,
\begin{equation}\label{eqn-poss}
s=\min_{j\in\Ck:e_j^T \Deltait x_\Ck < 0}\frac{e_j^T x_\Ck}{-e_j^T
  \Deltait x_\Ck}.
\end{equation}
If $\Deltait x_\Ck\not\le 0$, we
may compute the most limiting negative step for maintaining
nonnegativity of $x_\Ck + s \Deltait x_\Ck$, i.e.,
\begin{equation}\label{eqn-negs}
s=\max_{j\in\Ck:e_j^T \Deltait x_\Ck > 0}\frac{e_j^T x_\Ck}{-e_j^T
  \Deltait x_\Ck}.
\end{equation}
Since $\Deltait x_\Ck\ne 0$, at least one of (\ref{eqn-poss}) and
(\ref{eqn-negs}) is well defined. Pick one which is well defined, let
$j_k$ be a minimizing index and let $s_i$ be the corresponding
$s$-value. If $\Rkm=\Rk\backslash \{i_k\}$ and $\Ckm=\Ck\backslash
\{j_k\}$, then $A_{\Rkm\Ckm}$ is nonsingular, since $\Deltait x_\Ck$ up
to a scalar is the unique vector in the nullspace of $A_{\Rkm\Ck}$,
and $e_{j_k}^T\Deltait x_\Ck \ne 0$ by (\ref{eqn-poss}) and
(\ref{eqn-negs}).

Now assume that $x_j=0$ for some $j\in\Ck$. This is totally analogous
to the case $y_i=0$. Let $j_k=j$. Compute
$\Deltait y_\Rk$ as in Lemma~\ref{lem:invLP} for this $j$. If
$\Deltait y_\Rk\not\ge 0$, we may compute the most limiting positive
step for maintaining nonnegativity of $y_\Rk + t \Deltait y_\Rk$,
i.e.,
\begin{equation}\label{eqn-post}
t=\min_{i\in\Rk:e_i^T \Deltait y_\Rk < 0}\frac{e_i^T y_\Rk}{-e_i^T
  \Deltait y_\Rk}.
\end{equation}
If $\Deltait y_\Rk\not\le 0$, we
may compute the most limiting negative step for maintaining
nonnegativity of $y_\Rk + t \Deltait y_\Rk$, i.e.,
\begin{equation}\label{eqn-negt}
t=\max_{i\in\Rk:e_i^T \Deltait y_\Rk > 0}\frac{e_i^T y_\Rk}{-e_i^T
  \Deltait y_\Rk}.
\end{equation}
Since $\Deltait y_\Rk\ne 0$, at least one of (\ref{eqn-post}) and
(\ref{eqn-negt}) is well defined. Pick one which is well defined, let
$i_k$ be a minimizing index and let $t_j$ be the corresponding
$t$-value. If $\Rkm=\Rk\backslash \{i_k\}$ and $\Ckm=\Ck\backslash
\{j_k\}$, then $A_{\Rkm\Ckm}$ is nonsingular, since $\Deltait y_\Rk$ up
to a scalar is the unique vector in the nullspace of $A_{\Rk\Ckm}^T$,
and $e_{j_k}^T\Deltait y_\Ck \ne 0$ by (\ref{eqn-post}) and
(\ref{eqn-negt}).

Finally, we consider the case when $x_\Ck>0$ and $y_\Rk>0$. Then, for
any $i$, $j$ pair with $i\in\Rk$ and $j\in\Ck$, Lemma~\ref{lem:invLP}
gives $c_\Ck^T \Deltait x_\Ck> 0$ since $y_\Rk>0$ and $b_\Rk^T
\Deltait y_\Rk> 0$ since $x_\Ck>0$.  We may now pick an $i\in\Rk$ and
compute $\Deltait x_\Ck$ as in~\eqref{eq:deltaxy}. We must have
$\Deltait x_\Ck\ne 0$, since $e_i\ne 0$. We may now compute the most
limiting step from (\ref{eqn-poss}) or (\ref{eqn-negs}), out of which
at least one has to be well defined. Assume first the former. Let
$s_i$ denote the step and let $j$ be an index for which the minimum is
attained, so that $e_j^T(x_\Ck+s_i \Deltait x_\Ck)= 0$. By computing
$\Deltait y_\Rk$ for this $j$, there is an associated positive $t_j$
such that $e_i^T(y_\Rk+t_j \Deltait y_\Rk)= 0$ by
Lemma~\ref{lem:invLP}. If $y_\Rk+t_j \Deltait y_\Rk \ge 0$, we are
done. Otherwise, we let $t_j$ be the maximum positive step such that
$y_\Rk+t_j \Deltait y_\Rk \ge 0$. By construction, this must give a
strict reduction of $t_j$, but $t_j$ will remain strictly positive
since $y_\Rk>0$. We now conversely find the associated step
$s_i$. This process may be repeated a finite number of times for $i$,
$j$ pairs until $x_\Ck+s_i \Deltait x_\Ck \ge 0$ with $e_j^T(x_\Ck+s_i
\Deltait x_\Ck)= 0$ and $y_\Rk+t_j \Deltait y_\Rk \ge 0$ with
$e_i^T(y_\Rk+t_j \Deltait y_\Rk) = 0$. Note that (\ref{eqn-costsame})
of Lemma~\ref{lem:invLP} implies that each step gives a strict
reduction in objective function value, since one of the $s_i$ or $t_j$
values is reduced. Let $i_k=i$ and $j_k=j$. If $\Rkm=\Rk\backslash
\{i_k\}$ and $\Ckm=\Ck\backslash \{j_k\}$, then $A_{\Rkm\Ckm}$ is
nonsingular, since $\Deltait x_\Ck$ up to a scalar is the unique
vector in the nullspace of $A_{\Rkm\Ck}$, and $e_{j_k}^T\Deltait x_\Ck
\ne 0$. If (\ref{eqn-negt}) is used instead of (\ref{eqn-post}), the
argument is analogous, but now $s_i$ and $t_j$ are negative,
increasing towards zero.
\end{proof}

The optimality conditions given by Lemma~\ref{lem:partitionLP} imply
the existence of a nonsingular submatrix $A_{\Rp\Cp}$. Recursive
application of Lemma~\ref{lem-reducedimLP} gives a decomposition of
this matrix into square nonsingular submatrices of dimensions
shrinking by one row and one column at a time, corresponding to
primal-dual optimal pairs of $(PLP_k)$ and $(DLP_k)$ respectively, for
$k=r,r-1,\dots,1$. By reversing this argument, there must exist a
sequence of $r$ pivots $(i_1,j_1)$, $(i_2,j_2)$, $\dots$, $(i_r,j_r)$,
such that at stage $k$, optimal solutions to $(PLP_k)$ and $(DLP_k)$
are created, and at stage $r$, optimal solutions to $(PLP)$ and
$(DLP)$ are found. This is summarized in the following theorem.

\begin{theorem}\label{thm-LP}
  Assume that problems $(PLP)$ and $(DLP)$ both have feasible
  solutions.  For the optimality conditions given by
  Lemma~\ref{lem:partitionLP}, let $r=\abs{\Rp}=\abs{\Cp}$. Then,
  $r\le\min\{m,n\}$ and there are pairs of row and column indices
  $(i_k,j_k)$, $k=1,\dots,r$, which generate sets of row indices
  $R_1=\{i_1\}$, $\Rkp=\Rk\cup\{i_{k+1}\}$, and sets of column indices
  $C_1=\{j_1\}$, $\Ckp=\Ck\cup\{j_{k+1}\}$, with
  $i_{k+1}\in\{1,\dots,m\}\backslash \Rk$ and
  $j_{k+1}\in\{1,\dots,n\}\backslash \Ck$, such that for each $k$,
  $A_{R\k\Ck}$ is nonsingular, and $x_\Ck$ and $y_\Rk$ computed from
\[
A_{\Rk\Ck} x_\Ck = b_\Rk, \quad A_{\Rk\Ck}^T y_\Rk = c_\Ck, 
\]
are optimal to $(PLP_k)$ and $(DLP_k)$ respectively. In addition,
$x_{C_r}$ and $y_{R_r}$ together with $x_j=0$ for
$j\in\{1,\dots,n\}\backslash C_r$ and $y_i=0$ for
$i\in\{1,\dots,m\}\backslash R_r$ are optimal to $(PLP)$ and $(DLP)$
respectively.
\end{theorem}

\begin{proof}
  For $r=0$ or $r=1$, the result is immediate from the optimality
  conditions of Lemma~\ref{lem:partitionLP}. For $r\ge 2$, let
  $R_r=\Rp$ and $C_r=\Cp$, and repeatedly apply
  Lemma~\ref{lem-reducedimLP} for $k=r,r-1,r-2,...,2$. This gives an
  ordering of the indices of $\Rp$ and $\Cp$ as $i_r$, $i_{r-1}$,
  \dots, $i_1$ and $j_r$, $j_{r-1}$, \dots, $j_1$, such that the
  corresponding $x_\Ck$ and $y_\Rk$ are optimal to $(PLP_k)$ and
  $(DLP_k)$ respectively for $k=1,\dots,r$. In addition, $x_{C_r}$ and
  $y_{R_r}$ are optimal to $(PLP)$ and $(DLP)$ respectively. If the 
  ordering is reversed, so that $k=1,\dots,r$, the result follows.
\end{proof}

We note that Theorem~\ref{thm-LP} shows the existence of a sequence of
pivots that would create the optimal basis in $r$ steps, where
$r\le\min\{m,n\}$. We refer to such a sequence of pivots as a
\emph{short} sequence of pivots. This, however, does not constitute an
algorithm based on global information. We have no global information
on how to create the sequence of pivots. There is a straightforward
method given by enumerating all potential sequences of pivots that
generate primal-dual optimal pairs of linear programs of increasing
dimension.  However, we have not been able to give any useful bound on
the potential number of bases.

We also note that the short sequence of pivots does not explicitly
take into account objective function value, as the simplex method
does.  Therefore, we cannot ensure generating a short sequence of
pivots by making use of pivots associated with primal simplex only or
dual simplex only.

\section{Decomposing a matrix by convex combinations}

We also consider the related problem of decomposing a given $m\times
n$ matrix $M$ by convex combination, formulated as the following
primal-dual pair of linear programs
\begin{displaymath}
(P)\quad
  \begin{array}{ll}
     \minimize{u\in\Re^n,\alpha\in\Re}    & \alpha \\
     \subject & Mu + e \alpha \ge 0, \\
              & e\T u = 1, \\
              & u\ge 0,
  \end{array}
\qquad\quad
(D)\quad
  \begin{array}{ll}
     \maximize{v\in\Re^m,\beta\in\Re} & \beta \\
     \subject & M\T v +e \beta \le 0, \\
         & e\T v = 1, \\
         & v \ge 0.
  \end{array}
\end{displaymath}
Here, and throughout, $e$ denotes the vector of ones of the
appropriate dimension.

Problems $(P)$ and $(D)$ have a joint optimal value $\gamma$ by strong
duality for linear programming. The difference to the general linear
program is that $u$ and $v$ are defined on the unit simplex, and $(P)$
and $(D)$ are always feasible. This will enable a slightly stronger
result in which monotonicity in the objective function value may be
enforced.

We state the analogous results to the general linear programming case,
and point out the differences.

\begin{lemma}[Existence of optimal basic feasible 
  solution]\label{lem:partitionM} 
  For a given matrix $M$, there exists a number $\gamma$ and a
  partitioning of the row indices of $M$ into two sets $\Rp$ and   $\Rz$, and a partitioning of the column indices of $M$ into two sets
  $\Cp$ and $\Cz$, so that $\abs{\Rp}=\abs{\Cp}$, and associated with
  the resulting matrix
\[
\mtx{ccc}{ M_{\Rp \Cp} & M_{\Rp \Cz} & e \\
          M_{\Rz \Cp} & M_{\Rz \Cz} & e \\
             e^T     & e^T       & 0},
 \]
 the submatrix
\[
\mtx{cc}{ M_{\Rp \Cp} & e \\
             e^T     & 0}
 \]
is nonsingular, and there are vectors $u$ and $v$ for which
\begin{dbleqnarray*}
\mtx{cc}{ M_{\Rp\Cp} & e \\ e^T & 0 }
\mtx{c}{ u_{\Cp}\drop \\ \gamma } & = & \mtx{c}{0 \\ 1}, &
\mtx{cc}{ M_{\Rp\Cp}^T & e \\ e^T & 0 }
\mtx{c}{ v_{\Rp}\drop \\ \gamma } & = & \mtx{c}{0 \\ 1}, \\
\mtx{cc}{ M_{\Rz\Cp} & e }
\mtx{c}{ u_{\Cp}\drop \\ \gamma } & \ge & 0, &
\mtx{cc}{ M_{\Rp\Cz}^T & e }
\mtx{c}{ v_{\Rp}\drop \\ \gamma } & \le & 0, \\
       u_\Cp\drop & \geq & 0, &        v_\Rp\drop & \geq & 0, \\
       u_\Cz\drop & = & 0, &        v_\Rz\drop & = & 0,
\end{dbleqnarray*}
hold. The vectors $(u,\gamma)$ and $(v,\gamma)$ are optimal solutions
to $(P)$ and $(D)$ respectively.
\end{lemma}

\begin{proof}
  This is analogous to Lemma~\ref{lem:partitionLP}. The only
  difference is that $\alpha$ and $\beta$ are free variables.
\end{proof}

\begin{lemma}\label{lem:invM}
  Consider problems $(P)$ and $(D)$. Let $R$ denote a set of row
  indices of $M$ and let $C$ denote a set of column indices of
  $M$. Assume that
\begin{dbleqnarray*}
\mtx{cc}{ M_{RC} & e \\ e^T & 0 }
\mtx{c}{ u_{C}\drop \\ \alpha } & = & \mtx{c}{0 \\ 1}, &
\mtx{cc}{ M_{RC}^T & e \\ e^T & 0 }
\mtx{c}{ v_{R}\drop \\ \beta } & = & \mtx{c}{0 \\ 1}, \\
\mtx{cc}{ M_{RC} & e \\ e^T & 0 }
\mtx{c}{ \Deltait u_{C}\drop \\ \Deltait \alpha } & = & \mtx{c}{e_i \\ 0}, &
\mtx{cc}{ M_{RC}^T & e \\ e^T & 0 }
\mtx{c}{ \Deltait v_{R}\drop \\ \Deltait \beta } & = & \mtx{c}{e_j \\ 0}.
\end{dbleqnarray*}
Then,
\begin{subequations}
\begin{eqnarray}
\alpha & = & \beta, \\
        \Deltait \alpha & = & e_i^T v_R, \\
\Deltait \beta & = & e_j^T u_C, \\
e_j^T \Deltait u_C & = & e_i^T \Deltait v_R.
\end{eqnarray}
\end{subequations}
In addition, assume that $e_j^T \Deltait u_C = e_i^T \Deltait v_R\ne
0$. Then there are unique scalars $s_i$ and $t_j$ such that
\[
e_j^T (u_C+ s_i \Deltait u_C) = 0, \quad
e_i^T (v_R+ t_j \Deltait v_R) = 0,
\]
given by
\[
s_i=-\frac{e_j^T u_C}{e_j^T \Deltait u_C}, \quad t_j = -\frac{e_i^T
  v_R}{e_i^T \Deltait v_R}. 
\]
Furthermore,
\[
\alpha + s_i \Deltait \alpha = \beta + t_j \Deltait \beta.
\]
\end{lemma}

\begin{proof}
This is analogous to Lemma~\ref{lem:invLP}.
\end{proof}

As for the general linear programming case, associated with $(P)$ and
$(D)$ we will consider the linear programs
\begin{displaymath}
(P_k)\quad
  \begin{array}{ll}
     \minimize{u_\Ck\drop\in\Re^k,\alpha_k\in\Re}    & \alpha_k \\
     \subject & M_{\Rk\Ck}\drop u_\Ck\drop + e \alpha_k \ge 0, \\
              & e\T u_\Ck\drop = 1, \\
              & u_\Ck\drop\ge 0,
  \end{array}
\qquad\quad
(D_k)\quad
  \begin{array}{ll}
     \maximize{v_\Rk\drop\in\Re^k,\beta_k\in\Re} & \beta_k \\
     \subject & M_{\Rk\Ck}^T v_\Ck\drop +e \beta_k \le 0, \\
         & e\T v_\Ck\drop = 1, \\
         & v_\Ck\drop \ge 0.
  \end{array}
\end{displaymath}
where $\Rk$ denotes a set of row indices of $A$ and $\Ck$ denotes a
set of column indices of $M$ such that $\abs{\Rk}=\abs{\Ck}=k$.

\begin{lemma}\label{lem-reducegamma}
Let $M$ be an $m\times n$ matrix. Let $\Rk$ denote a set of row
indices of $M$ and let $\Ck$ denote a set of column indices of $M$
such that $\abs{\Rk}=\abs{\Ck}=k$, with $k\ge 2$. Assume that
\[
\mtx{cc} {M_{\Rk\Ck}\drop & e \\ e^T & 0}
\] 
is nonsingular, and assume that
\begin{displaymath}
 \mtx{cc} {M_{\Rk\Ck}\drop & e \\ e^T & 0} \mtx{c}{u_\Ck\drop \\ \gamma_k}
 = \mtx{c}{0 \\ 1}, 
\qquad
 \mtx{cc} {M_{\Rk\Ck}^T & e \\ e^T & 0} \mtx{c}{v_\Rk\drop \\
   \gamma_k} = \mtx{c}{0 \\ 1}, 
\end{displaymath}
where $u_\Ck\drop \geq 0$ and $v_\Rk\drop \geq 0$.

Then, there is a row index $i_k^{(1)}$, $i_k^{(1)}\in\Rk$, and a
column index $j_k^{(1)}$, $j_k^{(1)}\in\Ck$, such that
if $R_{k-1}^{(1)}=\Rk\backslash\{i_k^{(1)}\}$ and
$\Ckm=\Ck\backslash\{j_k^{(1)}\}$, then
\[
\mtx{cc} { M_{\Rkm\Ckm}^{(1)} & e \\ e^T & 0}
\] 
is nonsingular. Furthermore, it holds that
\[
 \mtx{cc} { M_{\Rkm\Ckm}^{(1)} & e \\ e^T & 0} \mtx{c}{u_\Ckm^{(1)} \\
   \gamma_{k-1}^{(1)}}  =  \mtx{c}{0 \\ 1}, 
\quad 
 \mtx{cc} { (M_{\Rkm\Ckm}^{(1)})^T & e \\ e^T & 0} \mtx{c}{v_\Rkm^{(1)} \\
   \gamma_{k-1}^{(1)}} = 
 \mtx{c}{0 \\ 1},
\]
for $u_\Ckm^{(1)}\ge0$, $v_\Rkm^{(1)}\ge0$  and $\gamma_{k-1}^{(1)} \leq \gamma_k^{(1)}$.

In addition, there is a row index $i_k^{(2)}$, $i_k^{(2)}\in\Rk$, and a column
index $j_k^{(2)}$, $j_k^{(2)}\in\Ck$, with $(i_k^{(2)},j_k^{(2)})\ne
(i_k^{(1)},j_k^{(1)})$, such that
if $R_{k-1}^{(2)}=\Rk\backslash\{i_k^{(2)}\}$ and
$\Ckm=\Ck\backslash\{j_k^{(2)}\}$, then
\[
\mtx{cc} { M_{\Rkm\Ckm}^{(2)} & e \\ e^T & 0}
\] 
is nonsingular. Furthermore, it holds that
\[
 \mtx{cc} { M_{\Rkm\Ckm}^{(2)} & e \\ e^T & 0} \mtx{c}{u_\Ckm^{(2)} \\
   \gamma_{k-1}^{(2)}}  =  \mtx{c}{0 \\ 1}, 
\quad 
 \mtx{cc} { (M_{\Rkm\Ckm}^{(2)})^T & e \\ e^T & 0} \mtx{c}{v_\Rkm^{(2)} \\
   \gamma_{k-1}^{(2)}} = 
 \mtx{c}{0 \\ 1},
\]
for $u_\Ckm^{(2)}\ge0$, $v_\Rkm^{(2)}\ge0$  and $\gamma_{k-1}^{(2)}
\geq \gamma_k^{(2)}$.
\end{lemma}

\begin{proof}
  The difference compared to Lemma~\ref{lem-reducedimLP} is that
  $\alpha_k$ and $\beta_k$ are free variables. Hence, there is no
  nonnegativity condition on them to handle. We note from
  Lemma~\ref{lem:invM} that 
\[
\mtx{cc}{ M_{\Rk\Ck} & e \\ e^T & 0 }
\mtx{c}{ \Deltait u_{\Ck}\drop \\ \Deltait \alpha_k } = \mtx{c}{e_i \\ 0}.
\]
Hence, we cannot have $\Deltait u_\Ck=0$, since $e \Deltait
\alpha_k=e_i$ cannot have a solution for $k\ge 2$. It follows that
$\Deltait u_\Ck\ne 0$ so that $e\T \Deltait u_\Ck= 0$ implies that
$\Deltait u_\Ck$ must have both strictly positive and strictly
negative components. The situation is analogous for $\Deltait
v_\Rk$. Therefore, there is a choice of selecting $s$ negative or
positive analogously to (\ref{eqn-poss}) and
(\ref{eqn-negs}). Consequently, there are two different possible index
pairs, one corresponding to $\Deltait \gamma_k\ge 0$ and one
corresponding to $\Deltait \gamma_k\le 0$.
\end{proof}
 
\begin{theorem}\label{thm-M}
  Let $M$ be a given $m\times n$ matrix. For the optimality conditions
  given by Lemma~\ref{lem:partitionM}, let
  $r=\abs{\Rp}=\abs{\Cp}$. Then, $r\le\min\{m,n\}$ and there are pairs
  of row and column indices $(i_k,j_k)$, $k=1,\dots,r$, which generate
  sets of row indices $R_1=\{i_1\}$, $\Rkp=\Rk\cup\{i_{k+1}\}$, and
  sets of column indices $C_1=\{j_1\}$, $\Ckp=\Ck\cup\{j_{k+1}\}$,
  with $i_{k+1}\in\{1,\dots,m\}\backslash \Rk$ and
  $j_{k+1}\in\{1,\dots,n\}\backslash \Ck$, such that for each $k$,
\[
\mtx{cc} { M_{\Rk\Ck}\drop & e \\ e^T & 0}
\] 
is nonsingular, and $(u_\Ck,\gamma_k)$ and $(v_\Rk,\gamma_k)$ computed from
\begin{dbleqnarray*}
\mtx{cc}{ M_{\Rk\Ck}\drop & e \\ e^T & 0 }
\mtx{c}{ u_{\Ck}\drop \\ \gamma_k } & = & \mtx{c}{0 \\ 1}, &
\mtx{cc}{ M_{\Rk\Ck}^T & e \\ e^T & 0 }
\mtx{c}{ v_{\Rk}\drop \\ \gamma_k } & = & \mtx{c}{0 \\ 1},
\end{dbleqnarray*}
are optimal to $(P_k)$ and $(D_k)$ respectively. In addition,
$(u_{C_r},\gamma_r)$ and $(v_{R_r},\gamma_r)$ together with $u_j=0$
for $j\in\{1,\dots,n\}\backslash C_r$ and $v_i=0$ for
$i\in\{1,\dots,m\}\backslash R_r$ are optimal to $(P)$ and $(D)$
respectively.

Such a sequence of pairs of row and column indices $(i_k,j_k)$,
  $k=1,\dots,r$, exists also if one of the additional requirements
  $\gamma_{k+1}\le \gamma_k$, $k=1,\dots,r-1$, or   $\gamma_{k+1}\ge
  \gamma_k$, $k=1,\dots,r-1$, are imposed.  
\end{theorem}

\begin{proof}
The result is analogous to Theorem~\ref{thm-LP}. The only difference
is the final statement on additional requirement $\gamma_{k+1}\le
\gamma_k$, $k=1,\dots,r-1$, or $\gamma_{k+1}\ge\gamma_k$,
$k=1,\dots,r-1$, not contradicting the existence of the short pivot
sequence. This is a consequence of the proved existence of two potential
pivots in the reduction step of Lemma~\ref{lem-reducegamma}.
\end{proof}

\section{Summary}

For a pair of linear programs in canonical forms that both are
feasible, we have shown the existence of a sequence of pivots of
length at most $\min\{m,n\}$ that leads from the origin to a
primal-dual pair of optimal solutions. At each step, the pivot creates
a nonsingular submatrix of the constraint matrix that increases in
dimension by one row an column by including the row and column of the
pivot element. By solving two linear systems involving the submatrix a
pair of primal and dual solutions are obtained. These solutions are
optimal for the restricted problem where only rows and columns of the
submatrix are included. At the final step, the solutions are optimal
to the full primal and dual problem respectively.

We have not been able to give rules for an algorithm taking into
account global information that would give this correct path without
potentially enumerating all possible paths, which might be
exponentially many. We therefore only publish the result as is, and
hope that the result will be useful for further understanding of
pivoting methods for linear programming.

We also note in passing that the reduction of
Lemma~\ref{lem-reducedimLP} and Lemma~\ref{lem-reducegamma} can be
done from an arbitrary basis matrix if the nonnegativity condition is
omitted. Therefore, there is a sequence of pivots of length at most
$\min\{m,n\}$ leading from any pair of primal-dual basic solutions to
the origin, Consequently, Theorem~\ref{thm-LP} and Theorem~\ref{thm-M}
imply that there is an overall bound of at most $2\min\{m,n\}$ pivots
leading from any pair of primal-dual basic solutions to an optimal
pair of primal-dual basic solutions. This is at least as tight as the
bound $n+m$ given by Fukuda and Terlaky~\cite{FT99}.

\section*{Acknowledgement}

We thank Krister Svanberg for many helpful discussions.

\bibliography{references}

\def\cprime{$'$}
\begin{thebibliography}{1}

\bibitem{BonM03}
T.~Bonates and N.~Maculan.
\newblock Performance evaluation of a family of criss-cross algorithms for
  linear programming.
\newblock {\em Int. Trans. Oper. Res.}, 10:53--64, 2003.

\bibitem{Dan63}
G.~B. Dantzig.
\newblock {\em Linear Programming and Extensions}.
\newblock Princeton University Press, Princeton, New Jersey, 1963.

\bibitem{FLN97}
K.~Fukuda, H.-J. Lüthi, and M.~Namiki.
\newblock The existence of a short sequence of admissible pivots to an optimal
  basis in {LP} and {LCP}.
\newblock {\em Int. Trans. Oper. Res.}, 4:273--284, 1997.

\bibitem{FT97}
K.~Fukuda and T.~Terlaky.
\newblock Criss-cross methods: A fresh view on pivot algorithms.
\newblock {\em Math. Prog.}, 79:369--395, 1997.

\bibitem{FT99}
K.~Fukuda and T.~Terlaky.
\newblock On the existence of a short admissible pivot sequence for feasibility
  and linear optimization problems.
\newblock {\em Pure Math. Appl.}, 10:431--447, 1999.

\bibitem{lemke1954}
C.~E. Lemke.
\newblock The dual method of solving the linear programming problem.
\newblock {\em Naval Research Logistics Quarterly}, 1(1):36--47, 1954.

\bibitem{Ter85}
T.~Terlaky.
\newblock A convergent criss-cross method.
\newblock {\em Optimization}, 16:683--690, 1985.

\bibitem{MR3100219}
R.~J. Vanderbei.
\newblock {\em Linear programming}, volume 196 of {\em International Series in
  Operations Research \& Management Science}.
\newblock Springer, New York, fourth edition, 2014.
\newblock Foundations and extensions.

\end{thebibliography}
\bibliographystyle{myplain}

\end{document}